\documentclass[12pt]{article}

\usepackage{amsthm}
\usepackage{amssymb}
\usepackage{amsmath}
\usepackage{amsfonts}
\usepackage{times}
\usepackage{enumerate}

\usepackage[paper=a4paper, top=2.5cm, bottom=2.5cm, left=2.5cm, right=2.5cm]{geometry}

\def\qed{\hfill $\vcenter{\hrule height .3mm
\hbox {\vrule width .3mm height 2.1mm \kern 2mm \vrule width .3mm
height 2.1mm} \hrule height .3mm}$ \bigskip}

\def \Sph{\mathbb{S}^{n-1}}
\def \RR {\mathbb R}
\def \Tr {\mathrm{Tr}}
\def \TR {\mathrm{Tr}}

\def \EE {\mathbb E}

\def \ZZ {\mathbb Z}

\def \Cov {\mathrm {Cov}}

\def \eps {\varepsilon}

\def \Id {\mathrm{Id}}

\def \Cov {\mathrm{Cov}}

\def \FF {\mathcal{F}}

\newtheorem{theorem}{Theorem}
\newtheorem{lemma}[theorem]{Lemma}

\newtheorem{example}{Example}

\newtheorem{proposition}[theorem]{Proposition}
\newtheorem{corollary}[theorem]{Corollary}
\theoremstyle{definition}

\theoremstyle{remark}
\newtheorem{remark}{Remark}

\long\def\symbolfootnotetext[#1]#2{\begingroup
\def\thefootnote{\fnsymbol{footnote}}\footnotetext[#1]{#2}\endgroup}

\begin{document}

\title{\Large Taming correlations through entropy-efficient measure decompositions with applications to mean-field approximation}
\date{}
\author{Ronen Eldan\thanks{Weizmann Institute of Science. Incumbent of the Elaine Blond Career Development Chair. Supported by a European Research Council Starting Grant (ERC StG) and by the Israel Science Foundation (grant No. 715/16).}}
\maketitle
\begin{abstract}
The analysis of various models in statistical physics relies on the existence of decompositions of measures into mixtures of product-like components, where the goal is to attain a decomposition into measures whose entropy is close to that of the original measure, yet with small correlations between coordinates. We prove a related general result: For every measure $\mu$ on $\RR^n$ and every $\eps > 0$, there exists a decomposition $\mu = \int \mu_\theta d m(\theta)$ such that $H(\mu) - \EE_{\theta \sim m} H(\mu_\theta) \leq \Tr(\Cov(\mu)) \eps$ and $\EE_{\theta \sim m} \Cov(\mu_\theta) \preceq \Id/\eps$. As an application, we derive a general bound for the mean-field approximation of Ising and Potts models, which is in a sense dimension free, in both continuous and discrete settings. In particular, for an Ising model on $\{\pm 1 \}^n$ or on $[-1,1]^n$, we show that the deficit between the mean-field approximation and the free energy is at most $C \frac{1+p}{p} \left ( n\|J\|_{S_p} \right)^{\frac{p}{1+p}}  $ for all $p>0$, where $\|J\|_{S_p}$ denotes the Schatten-$p$ norm of the interaction matrix. For the case $p=2$, this recovers the result of \cite{JKR18}, but for an optimal choice of $p$ it often allows to get almost dimension-free bounds.
\end{abstract}

\section{Introduction}
Given a probability measure $\mu$ on $\RR^n$, this work is concerned with the following question: Can we find a decomposition of $\mu$ as a mixture of measures $\mu = \int \mu_\theta d m(\theta)$ with the properties that, 
\begin{enumerate}[(i)]
\item	
The typical entropy of a measure $\mu_\theta$ is close to that of $\mu$,
\item	
The correlations between the coordinates of typical measure $\mu_\theta$ are tamed in the sense that the off-diagonal entries of the covariance matrix are small with respect to the diagonal ones.
\end{enumerate}

Several theorems of this type have been proven in the setting of the discrete hypercube and products of finite alphabets \cite{RT12,BCO16,CP17}. These types of decompositions have found applications to statistical mechanics \cite{BCO16,CKPZ16,CP17,JKR18} to rounding of semidefinite programs \cite{RT12,MR17} and to statistical estimation and inference \cite{Montanari2008EstimatingRV,CKPZ16}. 

As an example, let us formulate a direct corollary of a result proven in \cite{RT12}, which roughly states that for any product measure on the discrete hypercube $\{\pm 1 \}^n$, we may condition on $\ell$-coordinates so that, the conditional correlations are typically of order $O(1 / \sqrt{\ell})$. For $S \subset [n]$ and for $x = (x_1,...,x_n)$ we denote by $x_S$ the restriction of $x$ to the coordinates in $S$. Then, the result reads,

\begin{theorem} \label{thm:one} (\cite{RT12}, see also \cite{JKR18}).
	\label{thm:corr-rounding}
	Let $X = (X_1,\dots, X_n)$ be a random vector in $\{\pm 1\}^n$. Then, for any $\ell \in [n]$, there exists some $S \subset [n]$ with $|S| \leq \ell$ such that:
	$$\EE_{X_S} \EE_{\{u,v\} \in {V \choose 2}}\left[\Cov(X_u,X_v | X_S)^2\right] \le \frac{8\log 2}{\ell}.$$
\end{theorem}
\begin{remark}
The original formulation of this lemma bounds the mutual-information between $X_u$ and $X_v$ rather than the covariance. There also exists a version of the theorem which bounds the average mutual information between $k$-tuples of variables and can thus be used to bound higher order correlations, see \cite{JKR18}.
\end{remark}
\bigskip
\noindent The main result of the present work gives a bound of the same spirit, with several new features:
\begin{itemize}
\item
Our results hold for general measures in $\RR^n$, either continuous or discrete. In particular, we do not require the underlying space to have a product structure.
\item
Whereas existing results give a bound on the average entry of the covariance matrix, our theorem bounds the matrix in the positive definite sense (a bound on the average entry may then be obtained by taking traces).
\item
Our bounds allow for a "directional trade-off" between entropy and covariance, in a way that allows us to prioritize the directions in which we want the resulting covariance to be small. In our application to mean-field behavior of Potts models, we illustrate how this feature allows to improve the bounds in the literature.
\item
In Theorem \ref{thm:corr-rounding} as well as other existing results, the measures in the mixture are obtained by conditioning some of the coordinates to fixed values (called "pinning" in the statistical mechanics jargon). This type of conditioning makes no sense if there is no underlying product structure. Moreover, even if there is a product structure, in the continuous setting such an operation typically results in an infinite loss of entropy. Thus, our result uses a different type of conditioning, which roughly corresponds to exponential tilts of the original measure.
\end{itemize}
\bigskip
In order to formulate our results, we need some notation. We equip $\RR^n$ with a background measure $\nu$ (which will typically be either the Lebesgue measure or the uniform measure on $\{\pm 1\}^n$). For a measure $\mu$ which is absolutely continuous with respect to $\nu$, we define
$$
H_\nu(\mu) = - \int_{\RR^n} \log \left ( \frac{d \mu}{d \nu} \right ) d \mu
$$
whenever the integral converges. For a random variable $X \sim \mu$ we write $H_\nu(X) := H_\nu(\mu)$. When the background measure $\nu$ is clear from the context, we will simply abbreviate $H(\mu) = H_\nu (\mu)$. We also define the covariance matrix of $\mu$ by
$$
\Cov(\mu) := \int_{\RR^n} x^{\otimes 2} d \mu(x) - \left (\int_{\RR^n} x d \mu(x) \right )^{\otimes 2}.
$$

Our main result reads,
\begin{theorem}[Main decomposition theorem] \label{thm:main}
Let $\mu$ be a measure on $\RR^n$ such that its entropy $H_\nu(\mu)$ exists with respect to some background measure $\nu$. Then for every positive definite matrix $L \in \mathcal{M}_{n \times n}$, there exists a probability measure $m$ on some index set $\mathcal{I}$ and a family of 
probability measures $\{\mu_\theta \}_{\theta \in \mathcal{I}}$ on $\RR^n$ such that the measure $\mu$ admits the decomposition
\begin{equation}\label{eq:decomp}
\mu(A) = \int_{\mathcal{I}} \mu_\theta(A) d m(\theta), ~~ \forall A \subset \RR^n \mbox{ measurable},
\end{equation}
and such that the decomposition satisfies the following properties:
\begin{enumerate}
\item
\begin{equation}\label{eq:est-ent}
H_\nu(\mu) - \EE_{\theta \sim m} H_\nu(\mu_\theta) \leq \log \det (\Cov (\mu) L + \Id) \leq \Tr (\Cov(\mu) L),
\end{equation}
\item
\begin{equation}\label{eq:est2}
\EE_{\theta \sim m} \Cov(\mu_\theta) \preceq L^{-1},
\end{equation}
\item
\begin{equation}\label{eq:est1}
\EE_{\theta \sim m} \left [\Cov(\mu_\theta) L \Cov(\mu_\theta) \right ]  \preceq \Cov(\mu).
\end{equation}
\end{enumerate}
\end{theorem}	
\bigskip

The liberty of specifying the matrix $L$ which appears in the theorem allows us to control the trade-off between covariance and entropy: evidently, we should take the matrix $L$ to be large in the directions where we want the covariance to be small. In the next section, we show how a careful choice of this matrix allows to improve the bounds for mean-field approximation estimates. As a special case, taking the matrix $L$ in Theorem \ref{thm:main} to be a multiple of the identity recovers the same trade-off between entropy and covariance as in Theorem \ref{thm:one}. We formulate this special case as a corollary.

\begin{corollary}
Let $\mu$ be a measure on $\RR^n$ such that its entropy $H_\nu(\mu)$ exists with respect to some background measure $\nu$. For all $\eps>0$, there exists a  measure $m$ and a family of  probability measures $\{\mu_\theta \}_{\theta \in \mathcal{I}}$ such that the decomposition \eqref{eq:decomp} holds, and has the following properties.
\begin{itemize}
	\item $H_\nu(\mu) - \EE_{\theta \sim m} H_\nu(\mu_\theta) \leq \Tr(\Cov(\mu)) \eps$,
	\item $\EE_{\theta \sim m} \Cov(\mu_\theta) \preceq \frac{\Id}{\eps}$,
	\item $\EE_{\theta \sim m} \Tr( \left [\Cov(\mu_\theta)^2 \right ] ) \leq \frac{\Tr(\Cov(\mu))}{\eps}$.
\end{itemize} 
\end{corollary}
For a measure $\mu$ whose support has diameter of order $\sqrt{n}$, and in particular for a measure supported on $\{\pm 1\}^n$ we clearly have $\Tr(\Cov(\mu)) \leq n$. This immediately gives the following corollary, analogous to Theorem \ref{thm:corr-rounding}.
\begin{corollary} 
Let $X = (X_1,\dots, X_n)$ be a random vector in $\{\pm 1\}^n$. Then the random vector can be embedded in a probability space such that for any $\ell \in [n]$, there exists a $\sigma$-algebra $\Sigma$ which satisfies:
$$\EE \left [ \EE_{\{u,v\} \in {V \choose 2}}\left[\Cov(X_u,X_v | \Sigma )^2\right] \right ] \le \frac{1}{\ell}$$
and
$$
H(X) - H(X | \Sigma) \leq \ell.
$$
\end{corollary}

\noindent\textbf{Related work.} The realization that decomposition results such as Theorem \ref{thm:one} can be applied to statistical mechanics, in particular in order to derive the existence of Bethe states and compute free energies of random graph models, was done in a line of works by Bapst, Coja-Oghlan, Krzakala, Perkins and Zdeborov\'{a} \cite{BCO16,CKPZ16,CP17}. The remarkable insight that such decomposition results can also be used, in a very direct manner, to obtain bounds on mean-field approximations for Ising models appeared in the recent work of Jain, Koehler and Risteski \cite{JKR18}. 

One should point out that the method used in the paper \cite{JKR18} gives rise to a subexponential-time algorithm for calculating the partition function: The idea is that Theorem \ref{thm:one} can also be applied to pseudo-distributions, which allows an approximation by low-degree Sherali-Adams/Sum of Squares hierarchies. This builds on a previous work by Risteski \cite{pmlr-v49-risteski16}. When the model is ferromagnetic (or, alternatively, under some extra conditions on the magnitude of the interactions), a more efficient algorithm was suggested by Jain, Koehler and Mossel \cite{JKM17}.\\

In a related line of works it was established that, in the case of product spaces, if the measure $\mu$ satisfies a \emph{low-complexity} condition (which roughly requires the set of gradients of the log-density of the measure $\mu$ to have small Gaussian width or small covering numbers), it is possible to arrive at a stronger decomposition theorem which ensures proximity typical elements of the mixture to a product measure in transportation distance \cite{Eldan-GWComplexity, Austin-decomp, EldanGross-Decomp}. This complexity condition, when satisfied, is helpful in establishing large-deviation principles, following ideas of Chatterjee and Dembo \cite{CD14}. 

The framework of low-complexity measures was used by Basak and Mukherjee \cite{BM16} to obtain the first general bound for mean-field approximation for Potts models in terms of the Hilbert-Schmidt norm of the interaction matrix. In a recent preprint of Augeri \cite{Aug18}, a substantial quantitative improvement for the low-complexity framework was obtained, and as a result, a meaningful mean-field approximation bound holds true as long as the eigenvalues of the interaction matrix decay to zero, with a quantitative bound that also recovers the result of \cite{JKR18}.

\subsection{Application: Mean field behavior of Potts models on product spaces}

Two central models in statistical physics, which also have applications to computer science, statistics and learning theory, are the Ising and the Potts models. These are usually specified by a probability distribution on the discrete cube $\{\pm 1\}^n$ with a quadratic potential of the form
$$
f(\sigma) = \sum_{i,j} \sigma_i \sigma_j J_{i,j} + \sum_i h_i \sigma_i
$$
for $\sigma \in \{\pm 1\}^n$, $J$ some fixed symmetric interaction matrix, and $(h)_i$ a fixed {\it external field} vector. A generalization of this model is the Potts model, where the spins lie in some alphabet instead of the set $\{\pm 1\}$. A variant of those models appearing in the literature is continuous-spin version of those models, where $\{\pm 1\}$ is replaced by a real-valued spins, with the Hamiltonian having the same expression.

Our results will hold in the general setting of a Potts model with either discrete or continuous spins. To describe the model, we first set $k$ an integer and equip $\RR^k$ with some underlying measure $\nu$ (in the basic setting of a two-spin Ising model, $k=1$ and $\nu$ is the uniform measure on $\{\pm 1\}$). A spin $\sigma_i$ will be a point in $\RR^k$. The state of the system will be denoted by a point $\sigma \in (\RR^k)^n$ which assigns a spin to each particle. We fix an $n \times n$ symmetric matrix $J$ and an external field $h = (h_i)_{i \in [n]}$ with $h_i \in \RR^k$. The Hamiltonian of the system will have the form
\begin{equation}\label{eq:hamiltonian}
f(\sigma) = f_{J,h}(\sigma) := \sum_{j,i=1}^{n} J_{i,j} \sigma_i \cdot \sigma_j  + \sum_{i=1}^n h_i \cdot \sigma_i.
\end{equation}
 
Note that this framework also contains the Potts model over a finite alphabet, since we can choose the measure $\nu$ to be supported on the set of standard basis vectors $e_1,...,e_k$, in which case the expression $\sigma_i \cdot \sigma_j$ is equivalent to $\mathbf{1}_{\sigma_i = \sigma_j}$.

Two important associated quantities are the normalizing constant 
$$
Z = Z_{J,h} := \int \exp(f(\sigma)) d \nu^{\otimes n} (\sigma).
$$
which is also called the \emph{partition function} of the model and the quantity $\log{Z}$, called the \emph{free energy}. These quantities are clearly of significance for any task which requires access to the actual probabilities such as simulation of the models, but they are also important because several other basic properties of the model can be extracted from them by differentiation. For example, it is not hard to see that the mean of the $i$-th particle corresponds to $\frac{\partial \log Z}{ \partial h_i}$. Understanding global phenomena such as phase shifts also often boils down to the behavior of those quantities.

The task of calculating the partition function (either analytically or algorithmically) turns out to be notoriously hard in many basic examples of Ising and Potts models, and a whole theory has evolved around methods of approximation thereof (see \cite{BM16} and references therein). A central method relies on a variational approach referred to as the \emph{mean-field} approximation, on which we focus.

The starting point of mean-field approximations is the Gibbs variational principle which states that
$$
\log Z = \sup_{\tilde \mu} \left ( \int f(\sigma) d \tilde \mu(\sigma) + H_{\nu^{\otimes n}} (\tilde \mu) \right )
$$
where the supremum taken is over probability measures $\tilde \mu$ on $(\RR^k)^n$.

We say that the Hamiltonian $f$ admits a \emph{mean-field approximation} if the supremum in the above equation is attained by a product measure, up to a small error. More precisely, we say that the mean-field approximation of the model is within $\eps_n$ of the free energy if there exist measures $\xi_i$ on $\RR^k$, $i \in [n]$ such that the product measure $\xi = \xi_1 \otimes ... \otimes \xi_n$ satisfies
$$
\int f d \xi + H_{\nu^{\otimes n}} (\xi) \geq \log Z - \eps_n
$$
Evidently, the class of product measures is much simpler and more tractable. In the discrete case, mean field reduces the problem of approximating the free energy to optimizing a function over the set $[-1,1]^n$. Moreover, the mere existence of a mean-field approximation turns out to have implications regarding the behavior of the system (see e.g., \cite{BM16}).\\

In what follows, by slight abuse of notation, we will regard the matrix $J$ as an $(nk) \times (nk)$ matrix, in such a way that the right hand side of \eqref{eq:hamiltonian} is be replaced by the expression $\sum_{j,i=1}^{nk} J_{i,j} \sigma_i \sigma_j  + \sum_{i=1}^{nk} h_i \sigma_i$. This can clearly be done because the Hamiltonian $f$ is a polynomial of degree $2$ over $(\RR^k)^n$. In the case $k=1$, the expressions remain unchanged. \\

We give a sufficient condition for such an approximation to hold true. It improves the bounds of some existing results in the literature (\cite{BM16,Eldan-GWComplexity,pmlr-v75-jain18b,JKR18,Aug18}) and generalizes them to other settings.

\begin{theorem} \label{thm:meanfield}
There exists a product measure $\xi$ for which 
\begin{equation}\label{eq:main-mf}
\int f d \xi + H_{\nu^{\otimes n}} (\xi) \geq \log Z - 3 \log \det \left  (\Cov(\mu) \tilde J + \Id \right ).
\end{equation}
where $\tilde J := (J^2)^{1/2}$ is the matrix-absolute-value of $J$. 
\end{theorem}

The expression $\log \det (\Cov(\mu) \tilde J + \Id)$ may seem intractable at first, but as we will soon see, it can be efficiently bounded from above by Schatten-norms of the matrix $J$. Recall that the Schatten-$p$ norm of $J$ is defined as
$$
\| J \|_{S_p} := \left (\sum_{i \in [n]} |\lambda_i|^p \right )^{1/p},
$$
where $\lambda_1,...,\lambda_n$ are the eigenvalues of $J$. Moreover, a quantity of importance to us will be
$$
\mathcal{S}(J, S) := \max \left  \{ \sum_{i \in [n]} \log( \beta_i |\lambda_i| + 1); ~~ \sum_{i \in [n]} \beta_i \leq S \mbox{ and } \beta_i \geq 0, ~ \forall i \right \}.
$$
Then, we have the following bounds as corollaries of our theorem.

\begin{corollary} \label{cor:mf}
Suppose that the support of the measure $\nu$ has diameter bounded by $D > 0$ (a single spin lies in a bounded set). There exists a product measure $\xi$ for which 
\begin{equation}\label{eq:corr1}
\int f d \xi + H_{\nu^{\otimes n}} (\xi) \geq \log Z - 3 \mathcal{S}(J, D^2 n),
\end{equation}
where $\tilde J = (J^2)^{1/2}$. Moreover, for all $p>0$ one has
\begin{equation}\label{eq:schatten}
\int f d \xi + H_{\nu^{\otimes n}} (\xi)  \geq \log Z - 10 \frac{p + 1}{p} \left ( D^2 n \|J\|_{S_p} \right )^{\frac{p}{p+1}}.
\end{equation}
Finally, we also have
\begin{equation}\label{eq:rank}
\int f d \xi + H_{\nu^{\otimes n}} (\xi)  \geq \log Z - 3 \mathrm{Rank}(J) \log \left ( D^2 n \|J\|_{S_\infty} + 1 \right ).
\end{equation}
\end{corollary}

\begin{remark}
The result of \cite{JKR18} is identical to \eqref{eq:schatten} for the case that $\nu$ is uniform on $\{\pm 1\}$ and the choice $p=2$. Shortly after the first version of this paper appeared, Augeri \cite{Aug19} derived a bound of the order $\sqrt{n} \|J\|_{S_2}$ (which improves upon the bound of \cite{JKR18} and neither implies, nor is implied by, the present bound).
\end{remark}

We remark that the previous results in the literature give bounds which are necessarily polynomially large with respect to the dimension. In models where the eigenvalues of the interaction matrix exhibit a fast enough decay, our bounds will only be logarithmic in the dimension. We proceed with examples which illustrate this by showing how our bounds can be used to obtain a nearly optimal approximation in several cases. We begin with a prototypical example of a mean-field model, known as the Curie-Weiss model.

\begin{example} \emph{The Curie-Weiss model.} \\
We set $J_{i,j} = \frac{\beta}{n}$ for all $i \neq j$. This simple model is known to have either only two pure states (a positive and a negative state) or one state, depending on the magnitude of $\beta$. In this case we can clearly set $J_{i,i} = \frac{\beta}{n}$ without changing the distribution, which allows us to get $\mathrm{Rank}(J) = 1$. This, by equation \eqref{eq:rank}, gives a bound of $3 \log(n \beta)$ for the mean-field approximation. In the case of the critical temperature, $\beta=1$, this is known to be sharp up to the numerical constant.
\end{example}

\begin{example} \emph{Heat-kernel mesoscopic interactions on a lattice.} \\
We identify $[n]$ with the discrete torus $\mathbb{T} = \left (\ZZ / k \ZZ \right )^d$. We fix $0 < \alpha < 1$ and take the interaction matrix $J$ to be the kernel associated with the $\alpha k$-step random walk. In other words, suppose that $L$ is the matrix defined by the formula $L_{(x_1,...,x_d),(y_1,...,y_d)} = 2^{-d} \prod_{j \in [d]} \mathbf{1} \{|x_j - y_j| = 1 \mod k\}$ then we define $J = \beta L^{\alpha k}$ (for simplicity we assume that $\alpha k$ is an integer). Consider the Hamiltonian with interaction matrix $J$.

Expanding the operator $L$ in Fourier basis, it can be easily verified that the operator is diagonal with eigenvalue $\prod_{j \in [d]} \cosh(\omega_j / k)$ associated with the vector $x \to e^{2 \pi i \omega \cdot x / k }$ where $\omega = (\omega_1, ..., \omega_d)$. Thus, we get that the corresponding eigenvalue of the operator $J$ associated with the same vector is of the order $\beta \exp(-\alpha |\omega|^2 / 2)$.

A calculation then gives,
$$
\| J \|_{S_p} \asymp \beta \left ( \sum_{|\omega|^2 \leq kd} \exp(-\alpha p |\omega|^2 / 2) \right )^{1/p} \asymp  \beta \left ( \frac{1}{\alpha p} \right )^{d/p}
$$
Using the bound \eqref{eq:schatten}, we learn that the mean-field approximation error is at most of order
\begin{align*}
\frac{p + 1}{p} \left ( k^d \|J\|_{S_p} \right )^{\frac{p}{p+1}} ~& \leq  \frac{p + 1}{p} \left ( \beta k^d \left ( \frac{1}{\alpha p} \right )^{d/p} \right )^{\frac{p}{p+1}}.
\end{align*} 
Taking $p = \log(k)^{-1}$ we arrive at an approximation error of at most $C(d) \beta \log(k) \alpha^{-d}$. We conjecture that this bound is tight up to the $\log(k)$ term: Intuitively, each pure state depends roughly on the averages of neighborhoods of size $(k \alpha)^d$, which means that in order to choose a pure state one needs to determine the order of $\alpha^{-d}$ bits. 

We remark that there is nothing special about the particular choice of kernel; the same thing should be true for any smooth enough kernel.
\end{example}

\begin{example}\emph{Ising model on a $d$-regular expander.}\\
Let $G$ be a $d$-regular graph on $n$ vertices, such that the adjacency matrix of $G$ satisfies $\lambda_2(A_g) = O(\sqrt{d})$, with $\lambda(A)$ is the second largest eigenvalue of $A$, in absolute value, and with $A_G$ being the adjacency matrix. Consider the Ising model with interaction matrix $J = \frac{\beta}{d} A_G$.

In this case, we have for all $p>0$, $\|J\|_{S_p} = \beta \left (1 + \frac{n}{O(\sqrt{d})^p} \right )^{1/p}$. Taking $p = \frac{\log n}{\log d}$, we therefore have $\|J\|_{S_p} = O(1) n^{1/p}/\sqrt{d}$. The bound \eqref{eq:schatten} then gives that the error in the mean field approximation is bounded by
$$
\left ( n^{1+1/p} / \sqrt{d} \right )^{\frac{p}{p+1}} = \frac{n}{\sqrt{d}^{\frac{p}{p+1}}}.
$$
When $\log d \ll \log n$, we therefore have an error of $\frac{n}{\sqrt{d}^{1-o(1)} }$.

As expected, the mean-field approximation is only meaningful when $d \to \infty$. When $d$ is constant, meaningful approximations may be obtained by methods which rely on the tree-like structure of the graph. These are expressed in terms of the so-called Bethe free energy. See, e.g., \cite{DMS13, CP17} and references therein.
\end{example}

\section{Preliminaries and a stochastic construction} \label{sec:constr}

The proof of Theorem \ref{thm:main} relies on tools from stochastic calculus, specifically, the \emph{stochastic-localization} process used in several previous works, and first suggested in \cite{Eldan-thin-KLS} (also related to the Skorokhod embedding in \cite{eldan2016skorokhod}). The idea is to construct a process, driven by a Brownian motion, which samples from the measure $\mu$. In this section we describe the construction of the process and establish some properties which will be useful in our proof. The derivation of most of the properties needed for us, as well as the existence of the construction, have already been carried out in \cite{EMZ-CLT}. For the reader's convenience, we repeat the heuristic calculations needed for our proof, but we will often refer elsewhere for a more rigorous derivation. \\

In order to define our construction, we fix a standard Brownian motion $\{B_t\}_{t \geq 0}$ adapted to a filtration $\FF_t$ and a positive-definite matrix $Q$ which is a parameter of our construction. Our key definition will be a measure valued process $\{ \mu_t \}_{t \geq 0}$ which will be constructed via its density with respect to $\mu$, denoted by
$$
F_t(x) := \frac{d \mu_t}{d \mu} (x).
$$
Then, the process is defined via the equation
\begin{equation}\label{eq:stochastic}
F_0(x) = 1, ~~ d F_t(x) = F_t(x) (x - a_t) \cdot Q d B_t, ~~ \forall x \in \RR^n
\end{equation}
where $a_t = \int_{\RR^n} x F_t(x) \mu(dx)$. \\

The above is an infinite system of stochastic differential equations. However, a different point of view (to be seen later on), will reveal that it can be viewed as an It\^o process of a space of finite dimension. The existence of the process and some of its basic properties are summarized in the following proposition whose proof can be found in \cite{eldan2016skorokhod, EMZ-CLT}.

\begin{proposition} \label{prop:stochastic-localization}
	Equation \eqref{eq:stochastic} admits a unique solution and the measure-valued process $\mu_t$ has the following properties,
	\begin{enumerate}
		\item $\mu_0 = \mu$,
		\item $\mu_t$ is almost-surely a probability measure for all $t \geq 0$.
		\item For any continuous and bounded $\varphi : \RR^d \to \RR$, $\int_{\RR^d} \varphi(x) \mu_t(dx)$ is a martingale. \label{item:martingale}
	\end{enumerate}
\end{proposition}
\begin{proof} (sketch).
To see that Equation \eqref{eq:stochastic} is equivalent to a finite-dimensional It\^o process, observe that by It\^o's formula we have
$$
d \log F_t(x) = (x - a_t) \cdot Q d B_t - \frac{1}{2} |Q (x- a_t)|^2 dt.
$$
By integrating, one obtains,
$$
F_t(x) = \exp \left (   \int_0^t (x - a_s) \cdot Q d B_s - \frac{1}{2} \int_0^t |Q (x- a_s)|^2 ds \right )
$$
In other words, for all $t$ there exists $C_t \in \RR$ and $w_t \in \RR^n$ such that 
$$
F_t(x) = C_t \exp \left (   \langle w_t, x \rangle - \frac{1}{2} t |Q x|^2  \right ).
$$
Thus, equation \eqref{eq:stochastic} can be written in terms of the process $w_t$ which is finite-dimensional (this is done in detail in \cite{EMZ-CLT}). We turn to proving Properties 1-3.

Property 1 follows by definition. Property 2 follows from the calculation
$$
d \int \mu_t(dx) = \int d \mu_t(dx) = \int \left (F_t (x - a_t) \cdot Q d B_t \right ) \mu(dx) = \left ( \int x F_t(x) \mu(dx) - a_t \right ) \cdot Q d B_t = 0.
$$
Property 3 follows from from the fact that, by Equation \eqref{eq:stochastic}, $F_t(x)$ is a martingale for all $x \in \RR^n$ combined with a stochastic Fubini theorem.
\end{proof}

Our main theorem will rely on the analysis of the evolution of the covariance matrix of the measure $\mu_t$ as well as its entropy. To this end, we denote
$$
A_t := \Cov(\mu_t) = \int_{\RR^n} x ^{\otimes 2} \mu_t(dx) - a_t ^{\otimes 2}.
$$
Our goal is to calculate the differential of this process. We begin with,
\begin{align*}\label{eq:dat}
da_t ~& = d \int_{\RR^d} x \mu_t(dx)  \\
& = \int_{\RR^d} x d\mu_t(dx) \\
& = \left (\int_{\RR^d} x \otimes (x - a_t)\mu_t(dx) \right ) Q dB_t \\
& = \left (\int_{\RR^d} (x - a_t)^{\otimes 2} \mu_t(dx) \right ) Q dB_t \\ 
& = A_t Q dB_t.
\end{align*}

By Ito's isometry, the last equation implies that
$$
\frac{d}{dt} \EE \left [ a_t^{\otimes 2} \right ] = \EE [A_t Q^2 A_t].
$$

Since the quantity $\int_{\RR^d} x^{\otimes 2} \mu_t(dx)$ is a martingale (by Property 3 of Proposition \ref{prop:stochastic-localization}), we immediately obtain that for all $t>0$,
\begin{equation}\label{eq:dAt}
\frac{d}{dt} \EE \left [Q A_t Q \right ] = - Q \frac{d}{dt} \EE \left [ a_t^{\otimes 2} \right ] Q =  - \EE \left [ (Q A_t Q)^2 \right ] \preceq - \EE[QA_tQ]^2.
\end{equation}

\begin{lemma} \label{dHt}
For any background measure $\nu$ such that $H_\nu(\mu)$ is defined, we have
\begin{equation}\label{eq:dHt}
H_\nu(\mu) - \EE H_\nu(\mu_t) = \frac{1}{2} \int_0^t \EE \TR(Q A_s Q) ds.
\end{equation}
\end{lemma}
\begin{proof}
In the calculation, for an It\^{o} process $S_t$, we will use the convention
$$
d S_t = f_t dt + \mathrm{martingale},
$$
to specify that the process $S_t - \int_0^t f_s ds$ is a martingale. Since we assume that $H_\nu(\mu)$ is defined, we have that $\mu$ is absolutely continuous with respect to $\nu$. Denote $f(x) = \frac{d \mu}{d \nu}(x)$. Then,
\begin{align*}
H_\nu (\mu_t) ~& = - \int \log (f(x) F_t(x)) \mu_t(dx) \\
& = - \int \log f(x) \mu_t(dx) - \int \phi(F_t(x)) \mu(dx).
\end{align*}
where $\phi(x) = x \log x$. The first summand is a martingale according to Property 3 in Proposition \ref{prop:stochastic-localization}. Moreover, $F_t(x)$ is a martingale, which implies by It\^{o}'s lemma that
$$
\EE \left [ H_\nu(\mu) - H_\nu(\mu_t) \right ] = \frac{1}{2} \EE \left [ \int_{\RR^n} \left ( \int_0^t \phi''(F_t(x)) d [F(x)]_t \right ) d \mu(x) \right ]
$$
where $[F(x)]_t$ denotes the quadratic variation process of the martingale $F_t(x)$. Equation \eqref{eq:stochastic} gives that
$$
\frac{d[F(x)]_t}{dt} =  |Q(x-a_t)|^2 F_t(x)^2.
$$
Combining the last two displays yields,
\begin{align*}
H_\nu(\mu) - \EE \left [H_\nu(\mu_t) \right ] ~& = \frac{1}{2} \EE \left [ \int_{\RR^n} \left ( \int_0^t |Q(x-a_s)|^2 F_s(x) ds \right ) \mu(dx)  \right ] \\
& = \frac{1}{2} \EE \left [ \int_0^t \left ( \int_{\RR^n}  |Q(x-a_s)|^2 \mu_s(dx) \right  ) ds \right ] \\
& = \frac{1}{2} \int_0^t \EE \Tr(Q A_s Q) ds,
\end{align*}
which is the desired result.
\end{proof}

\section{Proof of the decomposition theorem}

This section is dedicated to the proof of Theorem \ref{thm:main}. Fix two probability measures $\mu, \nu$ on $\RR^n$ and a positive-definite matrix $L$. The decomposition will be given by the random measure $\mu_\eps$ obtained by the construction laid out in Section \ref{sec:constr}, with the choice $Q = L^{1/2}$. In other words, suppose that the Brownian motion is defined over an underlying probability space $(\Omega, \Sigma, m)$, also equipped with the corresponding filtration $\FF_t$. Let $\tau$ be a $\FF_t$-stopping time. Then $\mu_\tau$ is a measure-valued random variable over the probability space $\Omega$. For $\theta \in \Omega$ we then define $\mu_\theta = \mu_\tau(\theta)$. By Proposition \ref{prop:stochastic-localization}, Property 3, we have that
$$
\EE[\mu_\tau(A)] = \mu(A)
$$
for every measurable $A \subset \RR^n$, which is equivalent to the decomposition formula \eqref{eq:decomp}. Writing $A_t = \Cov(\mu_t)$ and keeping in mind the choice $Q^2 = L$, the proof of Theorem \ref{thm:main} thus amounts to establishing properties \eqref{eq:est-ent}, \eqref{eq:est2} and \eqref{eq:est1} which, may be written differently as
\begin{equation}\label{eq:est-ent-b}
H(\mu) - \EE H(\mu_\tau) \leq \log \det (Q^2 A_0 + \Id),
\end{equation}
\begin{equation}\label{eq:est2-b}
Q \EE A_\tau Q \preceq \Id,
\end{equation}
and
\begin{equation}\label{eq:est1-b}
\EE \left [ \left (Q A_\tau Q \right )^2 \right ] \preceq Q A_0 Q,
\end{equation}
respectively. 

Define the stopping time $\tau$ to be uniform in the interval $[1, 2]$, independent of the Brownian motion $B_t$. We begin with the bound \eqref{eq:est1-b}. By Equation \eqref{eq:dAt}, we deduce that $t \to \EE [Q A_t Q]$ is decreasing in the positive-definite sense, and thus
$$
Q A_0 Q \succeq \EE [Q A_1 Q] \succeq \EE [Q A_1 Q] - \EE[Q A_{2} Q] \stackrel{\eqref{eq:dAt}}{=} \EE \int_1^{2} \left (Q A_t Q\right )^2 dt = \EE\left [ \left (Q A_\tau Q \right )^2 \right ],
$$
where the second inequality uses the fact that $Q$ and $A_t$ are positive-definite, and the last equality uses the fact that $\tau$ is independent of the Brownian motion. The bound \eqref{eq:est1-b} is established.

We proceed to the bound \eqref{eq:est2-b}, for which we need the following lemma.

\begin{lemma} \label{lem:Gronwall}
Let $s > 0$. Let $f(t)$ be a function which is continuous on $[0,s]$, is differentiable in $(0, s)$ and which satisfies $f(0) > 0$ and
$$
f'(t) \leq - f(t)^2
$$
for all $t \in (0,s)$. Then we have,
$$
f(s) \leq \frac{1}{s + \tfrac{1}{f(0)}} \leq \frac{1}{s}.
$$
\end{lemma}
\begin{proof}
The result is a direct consequence of Gronwall's inequality and the fact that the function $t \to \frac{1}{t+a}$ solves the equation $f'(t) = - f(t)^2$.
\end{proof}

Equation \eqref{eq:dAt} and Jensen's inequality give that for all $\theta \in \Sph$, 
\begin{align*}
\frac{d}{dt} \EE[\langle \theta, Q A_t Q \theta \rangle ] ~& \leq - \langle \theta, \EE [QA_tQ]^2 \theta \rangle \leq - \EE [\langle \theta, QA_tQ \theta \rangle]^2.
\end{align*}

In conjunction with the above lemma, we learn that
\begin{equation}\label{eq:eigdecay}
\EE[\langle \theta, QA_tQ \theta \rangle] \leq \frac{1}{t + \frac{1}{\langle \theta, QA_0Q \theta \rangle}}, ~~ \forall \theta \in \Sph.
\end{equation}
Since $A_0$ is positive definite and since $\tau > 1$ almost surely and $\tau$ is independent of the process $A_t$, equation \eqref{eq:est2-b} follows. 

It remains to prove \eqref{eq:est-ent-b}. To that end, Equation \eqref{eq:dHt} with the fact that $\tau$ is independent from the process $A_t$ give
$$
H(\mu) - \EE[H(\mu_\tau)] = \frac{1}{2} \EE_\tau \left [ \int_0^\tau \EE[\Tr(QA_sQ)] ds \right ] \leq \frac{1}{2} \int_0^2 \EE[\Tr(QA_sQ)] ds.
$$
Let $u_1,..,u_n$ be a basis composed of unit eigenvectors of $QA_0Q$. The formula $\int_0^2 \frac{1}{t + \frac{1}{a}} d t = \log(2 a + 1)$ gives,
\begin{align*}
\int_0^2 \EE[\Tr(QA_tQ)] dt  ~& = \sum_{i \in [n]} \int_0^2 \EE[\langle u_i, QA_tQ u_i \rangle] dt   \\
& \stackrel{\eqref{eq:eigdecay}}{\leq} \sum_{i \in [n]} \int_0^2 \frac{1}{t + \langle u_i, QA_0Q u_i \rangle^{-1}} dt  \\
& =  \sum_{i \in [n]} \log \bigl (2 \langle u_i, Q A_0 Q u_i \rangle + 1 \bigr)  \\
& = \log \det (2 Q A_0 Q + \Id) \leq 2\log \det (Q A_0 Q + \Id).
\end{align*}
The conjunction of the last two displays gives \eqref{eq:est-ent-b} and completes the proof of Theorem \ref{thm:main}.

\section{The mean-field approximation}

In this section, we prove Theorem \ref{thm:meanfield} and Corollary \ref{cor:mf}.
\begin{proof}[Proof of Theorem \ref{thm:meanfield}]
	We invoke Theorem \ref{thm:main} with $L = \tilde J^{1/2} = (J^2)^{1/4}$, to obtain a measure $m$ on an index set $\mathcal{I}$, with 
	\begin{equation}\label{eq:Hest1}
	\EE_{\theta \sim m} \left (H(\mu) - H(\mu_\theta) \right ) \leq \log \det \left  (\tilde J^{1/2} \Cov(\mu) \tilde J^{1/2} + \Id \right )
	\end{equation}
	and
	$$
	\EE_{\theta \sim m} \tilde J^{1/2} \Cov(\mu_\theta) \tilde J^{1/2} \preceq \Id.
	$$
	Note that by the law of total variance, we have 
	$$
	\EE_{\theta \sim m} \tilde J^{1/2} \Cov(\mu_\theta) \tilde J^{1/2} \preceq \tilde J^{1/2} \Cov(\mu) \tilde J^{1/2}.
	$$ 
	Denoting the $u_1,...,u_n$ the unit eigenvectors of the matrix $\tilde J^{1/2} \Cov(\mu) \tilde J^{1/2}$ with $\lambda_1,...,\lambda_n$ being the corresponding eigenvalues, the two last inequalities give
	$$
	\EE_{\theta \sim m} \langle u_i, \tilde J^{1/2} \Cov(\mu_\theta) \tilde J^{1/2} u_i \rangle \leq \min(\lambda_i, 1) \leq 2 \log (1 + \lambda_i).
	$$
	Thus,
	\begin{equation}\label{eq:Hest2}
	\EE_{\theta \sim m} \Tr( \tilde J \Cov(\mu_\theta) ) \leq 2 \log \det (\tilde J^{1/2} \Cov(\mu) \tilde J^{1/2} + \Id).
	\end{equation}	
	
	For a measure $\rho$ on $(\RR^k)^n$, denote by $\xi(\rho)$ the unique measure such that its marginals on every $\RR^k$ are independent, and identical to the corresponding marginals of $\rho$. 
	
	Note that for a random vector $X = (X_1,...,X_{nk})$, if we define $\tilde X = (\tilde X_1, ..., \tilde X_{nk})$ such that $\tilde X_i$ are independent random variables and that $X_i$ and $\tilde X_i$ have the same distribution, then
	\begin{align*}
	\EE f(X) ~& = \sum_{i,j} J_{i,j} \EE[X_i X_j] + \sum_i h_i \EE[X_i] \\
	& = \sum_{i,j} J_{i,j} \left (\EE[X_i] \EE [X_j] + \Cov(X_i, X_j) \right ) + \sum_i h_i \EE[X_i] \\
	& = \EE f(\tilde X) + \Tr(J \Cov(X)) \leq \EE f(\tilde X) + \Tr(\tilde J \Cov(X)).
	\end{align*}
	This implies that $\int f d \rho \leq \int f d \xi(\rho) + \Tr(\tilde J \Cov(\rho))$ for all measures $\rho$. Moreover, since product distributions maximize entropy among the family of measures with prescribed marginals (when the background measure is a product measure), we have $H(\rho) \leq H(\xi(\rho))$. Combining those two facts, we have
	\begin{align*}
	\log \int \exp(f) d \nu^{\otimes n} ~& = \int f d \mu + H(\mu) \\
	& = \EE_{\theta \sim m} \left (\int f d \mu_\theta + H(\mu_\theta) - (H(\mu_\theta) - H(\mu))  \right ) \\
	& \leq \EE_{\theta \sim m} \left (\int f d \xi(\mu_\theta) + H(\xi(\mu_\theta)) \right ) + \EE_{\theta \sim m}  \Bigl  (H(\mu) - H(\mu_\theta) + \Tr(\tilde J \Cov(\mu_\theta)) \Bigr ) \\
	& \stackrel{ \eqref{eq:Hest1} \wedge \eqref{eq:Hest2} }{\leq} \EE_{\theta \sim m} \left (\int f d \xi(\mu_\theta) + H(\xi(\mu_\theta)) \right ) + 3 \log \det ( \tilde J^{1/2} \Cov(\mu) \tilde J^{1/2} + \Id).
	\end{align*}
	This shows the existence of $\theta \in \mathcal{I}$ for which the desired inequality holds true. The proof is complete.
\end{proof}

We move on to the proof of Corollary \ref{cor:mf}. The proof boils down to estimating the expression $\log \det ( \Cov(\mu) \tilde J + \Id)$, which is carried out in the two technical lemmas below.

\begin{lemma} \label{lem:annoyingbound}
Let $\alpha_i, \beta_i \geq 0$ for $i \in [n]$ with $\sum_{i \in [n]} \beta_i \leq S$ for some $S>0$. Then for all $p>0$,
\begin{equation}\label{eq:lemmax}
\sum_{i \in [n]} \log (\alpha_i \beta_i + 1) \leq \frac{3 p+1}{p} \left (S \| \alpha \|_p \right )^{p/(p+1)}.
\end{equation}
\end{lemma}
\begin{proof}
By reordering, assume without loss of generality that $\alpha_i$ is decreasing.  We use the inequality,
\begin{equation}\label{eq:sumbound}
\log(1+\alpha_i \beta_i) \leq \max(\log(\alpha_i \beta_i),0) + \min(\alpha_i \beta_i, 1)
\end{equation}
and bound each one of the terms separately, beginning with the first one.

Assume that the sequence $(\beta_i)_i$ maximizes the expression $\sum_{i \in [n]} \max(\log (\alpha_i \beta_i), 0)$ among all non-negative sequences whose sum is $S$; such maximum exists due to continuity and compactness. Suppose that the size of the support of the sequence $(\beta_i)$ is equal to $k$. We first claim that, without loss of generality we may also assume that the sequence is supported on the first $k$ entries and that
$$
\log (\alpha_i \beta_i) > 0, ~~ \forall i \in [k].
$$
Indeed, suppose that $\alpha_j \beta_j \leq 1$ for some $j \in [k]$. Then either $\alpha_i \beta_i \leq 1$ for all $i \in [k]$ in which case the sum is equal to zero, and otherwise if there is an index $\ell$ such that $\alpha_\ell \beta_\ell > 1$, then we can increase the expression $\sum_{i \in [k]} \max(\log(\alpha_i \beta_i),0)$ by increasing $\beta_\ell$ by $\eps$ and respectively decreasing $\beta_j$ by $\eps$, for $\eps$ small enough, contradicting the maximality of $(\beta_i)_i$.

It is therefore enough to bound from above the expression $\sum_{i \in [k]} \log(\alpha_i \beta_i)$. By concavity, we have that this expression is maximized for the choice $\beta_i = S/k$, for all $i \in [k]$, thus
$$
\sum_{i \in [k]} \log(\alpha_i \beta_i) \leq \max_k \left (\sum_{i \in [k]} \log \left ( \frac{\alpha_i S}{k} \right ) \right ).
$$
Now, by the generalized mean inequality, for all $p>0$,
$$
\sum_{i \in [k]} \log(\alpha_i) \leq k \log \left ( \frac{1}{k^{1/p}} \left (\sum_{i \in [k]} \alpha_i^p \right )^{1/p}  \right )
$$
and therefore
$$
\sum_{i \in [k]} \log \left ( \frac{\alpha_i S}{k} \right ) \leq k \log \left ( \frac{S \|\alpha\|_p }{k^{1+1/p}}   \right ) = \frac{p+1}{p} k \log \left ( \frac{ \left (S \|\alpha\|_p\right )^{\frac{p}{p+1}} }{k}   \right ).
$$
Using the inequality $x \log (a/x) \leq a/e$, we finally arrive at
\begin{equation}\label{eq:midbound}
\sum_{i \in [n]} \max(\log(\alpha_i \beta_i), 0) \leq \frac{p+1}{e p} \left (S \|\alpha\|_p\right )^{\frac{p}{p+1}}.
\end{equation}
We now move on to the term $\sum_{i \in [n]} \min(\alpha_i \beta_i, 1)$. Recalling that the sequence $(\alpha_i)$ is positive and decreasing, a moment of reflection reveals that under the constraint $\sum_i \beta_i = S$, this expression is maximized for the choice $\beta_i = 1/\alpha_i$ for all $i \leq k$, for some $k$, and $\beta_i = 0$ for all $i \geq k+1$, in which case $\sum_{i \in [n]} \min(\alpha_i \beta_i, 1) \leq k+1$. Now, we have
$$
\sum_{i \in [k]} \frac{1}{\alpha_i} = \sum_{i \in [k]} \beta_i \leq S
$$
and therefore, for all $p > 0$,
$$
\left (\frac{1}{k} \sum_{i \in [k]} \alpha_i^p \right )^{1/p}  \geq \left ( \frac{1}{k} \sum_{i \in [k]} \frac{1}{\alpha_i}\right )^{-1} \geq \frac{k}{S}.
$$
This implies $k \leq \left ( S \|\alpha\|_p \right )^{\frac{p}{p+1}}$, which gives 
$$
\sum_{i \in [n]} \min(\alpha_i \beta_i, 1) \leq 2 \left ( S \|\alpha\|_p \right )^{\frac{p}{p+1}}.
$$
Combining this inequality with \eqref{eq:sumbound} and \eqref{eq:midbound} finishes the proof.
\end{proof}

\begin{lemma} \label{lem:diag}
Let $A,B$ be two positive-definite matrices such that the eigenvalues of $A$ are $\alpha_1,...,\alpha_n$ and such that $\Tr(B) \leq S$. Then
$$
\log \det (AB + \Id) \leq \max \left \{ \sum_{i \in [n]} \log (\alpha_i \beta_i + 1); ~~ \sum_{i \in [n]} \beta_i = S \mbox{ and } \beta_i \geq 0, ~ \forall i \in [n]  \right \}.
$$
\end{lemma}
\begin{proof}
Without loss of generality, we can assume that $A$ is diagonal since $\det (AB + I) = \det ((U A U^{-1}) U B U^{-1} + I)$ and $\Tr(U B U^{-1}) = \Tr(B)$ for all orthogonal matrices $U$. Moreover, we remark that
$$
\log \det (AB + \Id) = \log \det (A^{1/2} B A^{1/2} + \Id).
$$
Now, let $B'$ and $B''$ be the diagonal and off-diagonal parts of $B$ respectively. Then,
$$
A^{1/2} B A^{1/2} + \Id = A^{1/2} B' A^{1/2} + \Id + A^{1/2} B'' A^{1/2}.
$$
Note that the matrix $A^{1/2} B' A^{1/2} + \Id$ is diagonal whereas the matrix $A^{1/2} B'' A^{1/2}$ is a symmetric matrix with zeros on its diagonal. Define,
$$
f(t) := \log \det \left ( A^{1/2} B' A^{1/2} + \Id + t A^{1/2} B'' A^{1/2} \right ).
$$
Since the determinant is a log-concave function on the positive-definite cone, and since both matrices $A^{1/2} B' A^{1/2} + \Id$ and $A^{1/2} B' A^{1/2} + \Id + A^{1/2} B'' A^{1/2}$ lie in the positive-definite cone, we conclude that the function $f(t)$ is concave on $[0,1]$. We claim that $f'(0) = 0$. Indeed, expanding the expression for the determinant, we remark that every permutation that makes a nonzero contribution is either the identity or has at least two non-fixed points. It follows that $f(t) = f(0) + O(t^2)$. In conjunction with the concavity of $f$, we learn that $f(1) \leq f(0)$. Thus, denoting the diagonal values of $B'$ by $\beta_1,...,\beta_n$, we have arrived at the inequality
$$
\log \det \left ( A B + \Id \right ) \leq \sum_{i \in [n]} \log(\alpha_i \beta_i + 1)
$$
with $\beta_i \geq 0$ for all $i \in [n]$ and $\sum_i \beta_i \leq S$. The proof is complete.
\end{proof}
	
We are now ready to prove the corollary.	

\begin{proof}[Proof of Corollary \ref{cor:mf}]
The combination of Equation \eqref{eq:main-mf} with Lemma \ref{lem:diag} teaches us that
$$
\int f d \xi + H_{\nu^{\otimes n}} (\xi) \geq \log Z - \mathcal{S}\bigl (J, \Tr(\Cov(\mu)) \bigr ).
$$
Now, since the diameter of the support of $\nu$ is $D$, we have that the diameter of the support of $\nu^{\otimes n}$ is $D \sqrt{n}$, which gives that $\Tr(\Cov(\mu)) \leq \int |x|^2 d \mu \leq D^2 n$. This proves \eqref{eq:corr1}.

An application of Lemma \ref{lem:annoyingbound} now teaches us that for all $p>0$,
$$
\mathcal{S}\bigl (J, D^2 n \bigr ) \leq 3 \frac{p+1}{p} \left (D^2 n \| \alpha \|_p \right )^{p/(p+1)},
$$
which proves equation \eqref{eq:schatten}.

Finally, to prove equation \eqref{eq:rank}, we note that if $\sum_i \beta_i \leq D^2 n$ and $(\alpha_i)_i$ are the eigenvalues of $\tilde J$, then
$$
\sum_{i \in [n]} \log(\alpha_i \beta_i + 1) \leq \mathrm{Rank}(J) \log \left (\max_{i \in [n]} \alpha_i \max_{i \in [n]} \beta_i + 1 \right ) \leq \mathrm{Rank}(J) \log \left (\|J\|_{OP} D^2 n + 1 \right ),
$$
as needed.
\end{proof}

\subsection*{Acknowledgements}
I'd like to thank Vishesh Jain, Frederic Koehler and Andrej Risteski for pointing out to me that the approximation in the example of the Curie-Weiss model is sharp, and for several other suggestions regarding the presentation of this manuscript. We are also thankful to Ofer Zeitouni and to the anonymous referee for useful comments and suggestions.

\bibliographystyle{alpha}
\bibliography{bib-decomp}

\end{document}